\documentclass[11pt,reqno,a4paper]{amsart}

\usepackage[latin1]{inputenc}

\usepackage[hidelinks]{hyperref}

\usepackage{etoolbox}
\usepackage{amsmath}
\usepackage{amssymb}
\usepackage{amscd}
\usepackage{amsthm}
\usepackage{amsfonts}
\usepackage{graphicx}
\usepackage[all,cmtip]{xy}
\usepackage{enumerate}

\patchcmd{\subsection}{-.5em}{.5em}{}{}
\patchcmd{\subsubsection}{-.5em}{.5em}{}{}

\usepackage[T1]{fontenc}

\usepackage[scaled]{helvet} 
\usepackage{courier} 
\usepackage{eulervm}
\normalfont

\makeatother
\usepackage{hyperref}

\numberwithin{equation}{section}

\newcommand{\cC}{\mathcal{C}}

\newcommand{\bR}{\mathbb{R}}

\newcommand{\bZ}{\mathbb{Z}}

\newcommand{\ra}{\rightarrow}

\newcommand{\qand}{\quad \textrm{and} \quad}

\newcommand\subsetsim{\mathrel{%
\ooalign{\raise0.2ex\hbox{$\subset$}\cr\hidewidth\raise-0.8ex\hbox{\scalebox{0.9}{$\sim$}}\hidewidth\cr}}}

\theoremstyle{theorem}
\newtheorem{theorem}{Theorem}[section]
\newtheorem{corollary}[theorem]{Corollary}
\newtheorem{proposition}[theorem]{Proposition}
\newtheorem{lemma}[theorem]{Lemma}

\newtheorem*{bd}{Borel Density Theorem}
\newtheorem*{sd}{Dani--Shalom Density Theorem}

\newtheorem*{mt}{Main Theorem}

\theoremstyle{definition}
\newtheorem{definition}[theorem]{Definition}
\newtheorem{remark}[theorem]{Remark}
\newtheorem{example}{Example}

\usepackage[all]{xy}
\usepackage{tikz-cd}

\renewcommand\labelenumi{(\roman{enumi})}
\renewcommand\theenumi\labelenumi 

\usepackage{mathtools}

\usepackage{changepage}
 \usepackage[margin=1in]{geometry} 

\begin{document}
\bibliographystyle{plain} 

\title{Borel density for approximate lattices}

\author{Michael Bj\"orklund}
\address{Department of Mathematics, Chalmers, Gothenburg, Sweden}
\email{micbjo@chalmers.se}

\author{Tobias Hartnick}
\address{Mathematisches Institut, Justus-Liebig-Universit\"at Gießen, Germany}
\email{tobias.hartnick@math.uni-giessen.de}

\author{Thierry Stulemeijer}
\address{Mathematisches Institut, Justus-Liebig-Universit\"at Gießen, Germany}
\email{thierry.stulemeijer@math.uni-giessen.de}

\keywords{Approximate lattices, approximate groups, Borel density, ergodic joinings}

\subjclass[2010]{Primary: 22D40 ; Secondary: 20P05, 20G25}

\maketitle
\begin{abstract} We extend classical density theorems of Borel and Dani--Shalom on lattices in semi-simple, respectively solvable algebraic groups over local fields to approximate lattices.  Our proofs are based on the observation that Zariski closures of approximate subgroups are close to algebraic subgroups. Our main tools are stationary joinings between the hull dynamical systems of discrete approximate subgroups and their Zariski closures.
\end{abstract}

\section{Introduction}
Borel's density theorem \cite{Borel} is a cornerstone of the theory of lattices in semisimple algebraic groups over local fields, and can be stated as follows.
\begin{bd}
Let $k$ be a local field and let ${\bf G}$ be a connected semisimple algebraic group over $k$. If ${\bf G}(k)$ does not have any compact factors, then every lattice $\Gamma<{\bf G}(k)$ is Zariski dense.
\end{bd}
Here and in the sequel, ${\bf G}(k)$ is considered as a topological group with respect to its natural Hausdorff group topology, which turns ${\bf G}(k)$ into a locally compact second countable (lcsc) topological group. A similar density theorem for lattices in solvable algebraic groups was etablished by Dani \cite{Dani} (for $k=\bR$) and Shalom \cite{Shalom} (for general local fields). Here a solvable algebraic group ${\bf G}$ over a field $k$ is called \emph{$k$-split} if every composition factor is isomorphic over $k$ to the additive or multiplicative group of $k$. For example, unipotent algebraic groups over fields of characteristic $0$ are $k$-split, but they need not be $k$-split in positive characteristic (see Subsection \ref{SecUnipotent}).
\begin{sd}
Let $k$ be a local field and let ${\bf G}$ be a connected solvable algebraic group over $k$. If ${\bf G}$ is $k$-split, then every lattice $\Gamma < {\bf G}(k)$ is Zariski dense.
\end{sd}
In this article we generalize these density theorems to (certain) approximate lattices. Approximate lattices are certain discrete approximate subgroups (in the sense of Tao \cite{Tao08}) of locally compact groups. They were introduced in \cite{BH1} as generalizations of Meyer sets in abelian lcsc groups \cite{Meyer} and further studied in \cite{BHP1, BH2, BH3, Machado}. By definition, a \emph{uniform approximate lattice} $\Lambda \subset G$ is a discrete approximate subgroup which is  cocompact in the sense that $G = \Lambda K$ for a compact subset $K\subset G$. More generally, \emph{approximate lattices} are defined by the existence of non-trivial stationary measures on an associated hull dynamical system,  and \emph{strong approximate lattices} are those approximate lattices, for which the hull even admits an invariant measure (see Section \ref{SecPrelim} for precise definitions). With this terminology understood, our main result can be stated as follows:
\begin{mt} Let $k$ be a local field and let ${\bf G}$ be a connected algebraic group over $k$. Assume that either ${\bf G}$ is semisimple and ${\bf G}(k)$ does not have any compact factors, or that ${\bf G}$ is solvable and $k$-split. Then every strong approximate lattice $\Lambda \subset G$ and every uniform approximate lattice $\Lambda \subset G$ is Zariski dense.
\end{mt}
By definition, a subgroup $\Gamma<G$ is a uniform approximate lattice iff it is a uniform lattice. We will show in Subsection \ref{SecNonUniform} below that a subgroup $\Gamma<G$ is an approximate lattice iff it is a strong approximate lattice iff it is a lattice. Thus our main theorem is indeed a proper generalization of the classical density theorems.

Our proof is inspired by Furstenberg's proof of Borel density \cite{Furstenberg}, which can be sketched as follows: If $\Gamma$ is a lattice in $G= {\bf G}(k)$ and $H$ denotes the Zariski closure of $\Gamma$ in $G$, then the invariant probability measure on $G/\Gamma$ pushes forward to an invariant probability measure on $G/H$, which by Chevalley's theorem can be realized as a quasi-projective variety. Using reccurence properties of unipotents on projective space with respect to the invariant measure at hand one then deduces that $G/H$ must be a point.

This approach does not apply directly to our more general setting for several reasons: Firstly, the Zariski closure of an approximate lattice $\Lambda \subset G$ is not a group. It is however, in a sense made precise in Theorem \ref{AASMain} below, close to an algebraic subgroup $H$ of $G$. We would thus like to connect a stationary measure on the hull of $\Lambda$ (which serves as a natural replacement for the homogeneous space $G/\Gamma$ in the group case) to a measure on $G/H$. Unlike the group case, we cannot embed the hull of $\Lambda$ into $G/H$, but we can use a stationary joining between the hull and $G/H$ to obtain a measure on $G/H$. A crucial difference to the group case will be that the measure obtained on $G/H$ will in general not be invariant, but only stationary. To obtain the desired conclusion, we thus need to investigate further properties of the measure in question. In this final step we also need information concerning maximal algebraic subgroups of semisimple groups over local fields as provided by Stuck \cite{St}. 

This article is organized as follows: In Section \ref{SecPrelim} we recall the precise definitions of strong and uniform approximate lattices. We use this opportunity to establish a number of basic results concerning hull dynamical systems, which will be used throughout the article. In Section \ref{SecAlgebraicApproximateGroups} we show that Zariski closures of approximate subgroups are again approximate subgroups, and that such ``algebraic'' approximate subgroups are close to algebraic subgroups. This statement is made precise in Theorem \ref{AASMain}, which is the main result of this section. In Section \ref{SecMain} we use this result to deduce Borel density, first in the uniform case, and then in the strong case. In Section \ref{SecVariants} we derive Dani--Shalom density and discuss various variants and refinements.

 Appendix \ref{AppProjection} contains some background concerning the existence of stationary joinings. Appendix \ref{AppUnimodular} generalizes the unimodularity theorem from \cite{BH1} to the case of non-compactly generated groups; this is used in the proof of the main theorem in the uniform case.

Throughout this article we use the following convention: If $k$ is a local field and ${\bf G}$ is a linear algebraic group over $k$, then all topological terms (e.g.\ closure, compactness etc.) concerning subsets of $G := {\bf G}(k)$ refer to the Hausdorff topology on $G$ and not to the Zariski topology, unless explicitly mentioned otherwise.

{\bf Acknowledgement.} We are indebted to Oliver Sargent for challenging us to establish Borel density for approximate lattices. We thank JLU Gie\ss en for providing excellent working conditions during several visits of the first-named author. 

\section{Approximate lattices and their hulls}\label{SecPrelim}

\subsection{Uniform approximate lattices}
Let $G$ be a group. Given subsets $A, B \subset G$ we denote by 
\[
AB := \{ab\in G \mid a \in A, b \in B\} \qand A^{-1} := \{a^{-1} \in G \mid a \in A \} 
\]
their product set, respectively set of inverses. We also define $A^{k} := A^{k-1}A$ for $k \geq 2$. (For distinction we write $A^{\times k}$ for the $k$-fold Cartesian product.)

We recall that if $G$ is a group, then a subset $\Lambda \subset G$ is called an \emph{approximate subgroup} if it is symmetric, contains the identity and satisfies 
$\Lambda^2 \subset F\Lambda$ for a finite subset $F \subset G$. Since $\Lambda$ is symmetric, this implies $\Lambda^2 \subset \Lambda F^{-1}$, and hence we may choose a finite set $F_\Lambda$ such that
$\Lambda^2 \subset F_\Lambda \Lambda \cap \Lambda F_\Lambda$, and hence
\begin{equation}\label{FLambda}
\Lambda^k \subset F_\Lambda^{k-1} \Lambda \cap \Lambda F_\Lambda^{k-1}\quad \text{for all }k \geq 2.\end{equation}

If $G$ is a lcsc group, then a subset $P \subset G$ is called \emph{(left-)relatively dense} if there exists a compact subset $K \subset G$ such that $G = PK$. It is called \emph{uniformly discrete} if $e$ is not an accumulation point of $P^{-1}P$. An approximate subgroup $\Lambda \subset G$ is called a \emph{uniform approximate lattice}, if it is relatively dense and discrete. In this case, also $\Lambda^k$ is a uniform approximate lattice for all $k \in \mathbb N$, and in particular $\Lambda$ is uniformly discrete. 

Note that a subgroup of $G$ is a uniform approximate lattice if and only if it is a uniform lattice. We now proceed towards the definition of non-uniform approximate lattices, which generalize non-uniform lattices in a similar way.

\subsection{The Chabauty-Fell topology}

Given a lcsc space $X$ we denote by $\mathcal C(X)$ the collection of closed subsets of $X$ with the Chabauty-Fell topology, i.e.\ the topology on $\mathcal C(X)$ generated by the basic open sets 
\[
U_K = \{A \in \mathcal C(X) \mid A \cap K = \emptyset\} \quad \text{and} \quad U^{V} = \{A \in \mathcal C(X) \mid A \cap V \neq \emptyset\},
\]
where $K$ runs over all compact subsets of $X$ and $V$ runs over all open subsets of $X$. 

Under the present assumptions on $X$, the space $\mathcal C(X)$ is a compact metrizable space (see e.g. \cite[Prop. 1.7 and Prop. 1.8]{Paulin}), and in particular its topology is characterized by convergence of sequences in $\mathcal C(X)$. A sequence $ (F_i) $ in $ \mathcal{C}(X) $ converges if and only if the two following conditions are satisfied:
\begin{enumerate}[(CF1)]
\item For all $ x\in F $ there exist $ x_i \in F_i $ such that $ (x_i) $ converges to $ x $.
\item If $x_i \in F_i$ for all $i \in \mathbb N$ then every accumulation point of the sequence $(x_i)$ is contained in $F$.
\end{enumerate}
We derive two consequences: Firstly, if a $G$ acts jointly continuously on $X$, then it acts jointly continuously on $\mathcal C(X)$ by $g.A := \{ga\mid a \in A\}$. Secondly, taking finite unions is continuous in the Chabauty--Fell topology: 
\begin{corollary}\label{Cor:continuity of the union map}
For every lcsc space $X$ the map $ \pi \colon \mathcal{C}(X)^{\times k}\to \mathcal{C}(X)$,  $(F_1,\dots ,F_k)\mapsto F_1\cup \dots \cup F_k$ is continuous.
\end{corollary}
\begin{proof}
Let $ ((F_{1,i},\dots ,F_{k,i}))_{i\geq 1} $ be a sequence in $ \mathcal{C}(G)^{\times k} $ converging to $ (E_1,\dots ,E_k) $, and abbreviate $F_i := F_{1,i}\cup \dots \cup F_{k,i}$ and $E :=  E_1\cup \dots \cup E_k$. We have to show that $F_i \to E$; for this we check Conditions (CF1) and (CF2):

(CF1) If $x \in E$, then $x \in E_j$ for some $j \in \{1, \dots, k\}$. Since $F_{j,i} \to E_j$ there thus exist $ x_i\in F_{j,i} \subseteq F_i $ such that $x_i \to x$.

(CF2) Let $x_i\in F_i$ and let $x \in X$ be an accumulation point of $(x_i)$, say $x_{n_i} \to X$. Passing to a further subsequence we may assume by the pigeon hole principle that $x_{n_i} \in F_{j, n_i}$ for some $j \in \{1, \dots, k\}$. Since $F_{j, i} \to E_j$ it then follows that $x \in E_j \subset E$.
\end{proof}

\subsection{Hulls of closed subsets}
Let $G$ be a lcsc group. We refer to a jointly continuous action of $G$ on a compact space $\Omega$ as a \emph{topological dynamical system (TDS)} and to a continuous $G$-equivariant map between TDSs as a factor map. If $G \curvearrowright \Omega$ is a TDS, then so ist the orbit closure of every element of $\Omega$, and every factor map maps an orbit closure of an element onto the orbit closure of its image. 

By the results recalled in the previous subsection, the left-action of a lcsc group $G$ on itself induces a TDS $G \curvearrowright \mathcal C(G)$, $(g,A) \mapsto gA$, and more generally the diagonal action of $G$ on $G^{\times n}$ induces a TDS $G \curvearrowright \mathcal C(G)^{\times n}$, $(g, (A_1, \dots, A_n)) \mapsto (gA_1,\dots, gA_n)$. We are going to consider orbit closures in these TDSs.
\begin{definition} Let $G$ be a lcsc group and let $P, P_1, \dots, P_n \in \mathcal C(G)$.
\begin{enumerate}[(i)]
\item The \emph{(left-)hull} of a closed subset $P \subset G$ is defined as the orbit closure
\[
\Omega_P := 
\overline{\{gP \mid g \in G\}} \subset \mathcal C(G).
\]
\item The \emph{simultaneous (left-)hull} of $P_1, \dots, P_n$ is
\[
\Omega_{P_1, \dots, P_n} := \overline{\{(gP_1, \dots, gP_n) \mid g \in G\}} \subset \mathcal C(G)^{\times n}.
\]
\end{enumerate}
\end{definition}
In the case of closed subgroup $H<G$, the hull is a compactification of $G/H$, but it turns out to be the trivial compactification:
\begin{lemma}[Hulls of closed subgroups]\label{lemma2} If $H<G$ is a closed subgroup, then $\Omega_H \setminus \{\emptyset\} = G/H$.
\end{lemma}
\begin{proof} Let $(g_n)$ be a sequence in $G$ such that $g_n H$ converges to some $H' \in \mathcal C(G)$ and assume $H' \neq \emptyset$. Then there exists $x \in H'$ and by (CF1) there exist $h_n \in H$ such that $g_nh_n \to x$. In particular there exists a compact set $K$ such that $k_n := g_nh_n \in K$ and $k_n H = g_n H \to H'$. Passing to a subsequence we may assume that $k_n$ converges to some $k \in K$, and by continuity of the $G$-action we deduce that $k_n H \to kH$. Thus $H' = kH \in G/H$.
\end{proof}
Note that if $P_1, \dots, P_n \in \mathcal C(G)$, then the projection onto the $i$th factor yields a continuous surjective $G$-factor map $\pi_i: \Omega_{P_1, \dots, P_n} \to \Omega_{P_i}$. If we set $P := P_1 \cup \dots \cup P_N$, then by Lemma \ref{Cor:continuity of the union map} we also have a continuous $G$-factor map $\pi: \Omega_{P_1, \dots, P_n} \to \Omega_P$ given by $(Q_1, \dots, Q_n) \mapsto Q_1 \cup \dots \cup Q_n$.

We apply these factor maps to study relatively dense subsets of lcsc groups. This is made possible by the observation that if $P \in \mathcal C(G)$, then $\emptyset \in \Omega_P$ if and only if $P$ is not relatively dense \cite[Prop. 4.4]{BH1}. We will use the fact that every TDS contains a \emph{minimal} subset, i.e.\ a subset which is the orbit closure of each of its elements. Note that if a minimal system $Z$ contains a fixpoint $p$, then $Z = \{p\}$. 
\begin{lemma}
\label{lemma_unformindex}
Let $P_1,\ldots,P_k \in \cC(G)$. If $P := P_1 \cup \ldots \cup P_n$ is relatively dense then there exists $i \in \{1,\ldots,n\}$
such that ${P_i^{-1}P_i}$ is relatively dense. 
\end{lemma}
\begin{proof} We choose a minimal subset $Z$ of the joint hull $\Omega_{P_1, \dots, P_n}$ and an element $z \in Z$. We then set $Q := \pi(z)$ and $Q_i := \pi_i(z)$, where
$\pi_i: \Omega_{P_1, \dots, P_n} \to \Omega_{P_i}$ and $\pi: \Omega_{P_1, \dots, P_n} \to \Omega_P$ are the factor maps defined above. 

We first observe that $z \neq (\emptyset, \dots, \emptyset)$, since otherwise $Q = \emptyset$, contradicting the assumption that $P$ be relatively dense. There thus exists $i \in \{1, \dots, n\}$ such that $Q_i \neq \emptyset$. Note that $\emptyset$ is a $G$-fixpoint in $\mathcal C(G)$. Since $\Omega_{Q_i} = \pi_i(Z)$ is minimal and $Q_i \neq \emptyset$, we thus deduce that $\emptyset \not \in \Omega_{Q_i}$, hence $Q_i$ is relatively dense in $G$. Let $K_1 \subset G$ be a compact subset such that $G = Q_iK_1$.

Since $Q_i \in \Omega_{P_i}$, we deduce from \cite[Lemma 4.6]{BH1} that $Q_i^{-1}Q_i \subset \overline{P_i^{-1}P_i}$, and hence $Q_i^{-1}Q_i \subset P_i^{-1}P_iK_0$, where $K_0$ is any compact identity neighbourhood in $G$. We thus obtain
\[
G \subset Q_i^{-1}Q_iK_1 \subset P_i^{-1}P_iK_0K_1
\]
which shows that $P_i^{-1}P_i$ is relatively dense in $G$.
\end{proof}

\subsection{Non-uniform approximate lattices}\label{SecNonUniform}
We now turn to the definition of non-uniform approximate lattices. Let $G$ be a lcsc group and let $\Gamma< G$ be a discrete subgroup. By Lemma \ref{lemma2} we have $\Omega_\Gamma \setminus \{\emptyset\} = G/\Gamma$. Thus $\Gamma$ is a lattice in $G$ if and only if there exists a $G$-invariant probability measure on $\Omega_\Gamma \setminus \{\emptyset\}$. Equivalently, if $\mu$ is any admissible probability measure on $\mu$ (i.e. absolutely continuous with respect to Haar measure and with support generating $G$ as a semigroup), then there exists a $\mu$-stationary probability measure on $\Omega_\Gamma \setminus \{\emptyset\}$. 
\begin{definition}\label{DefSAL} Let $G$ be a lcsc group, let $P \in \mathcal C(G)$ and let $\Lambda \subset G$ be a closed and discrete approximate subgroup.
\begin{enumerate}[(i)]
\item A probability measure $\nu$ on $\Omega_P$ is called \emph{non-trivial} if $\nu(\{\emptyset\}) = 0$.
\item $\Lambda$ is called a \emph{strong approximate lattice} if there exists a non-trivial $G$-invariant probability measure $\nu$ on $\Omega_\Lambda$.
\item $\Lambda$ is called an \emph{approximate lattice} if for every admissible probability measure $\mu$ on $G$ there exists a non-trivial $\mu$-stationary probability measure $\nu$ on $\Omega_\Lambda$.
\end{enumerate}
\end{definition}
Concerning the relations between these definitions we remark: Every uniform approximate lattice and every strong approximate lattice is an approximate lattice. We do not currently know whether every approximate lattice is strong. If $G$ is amenable, then every uniform approximate lattice is strong, and if $G$ is nilpotent, then every approximate lattice is both strong and uniform (\cite[Thm. 4.25]{BH1}). A discrete subgroup $\Gamma < G$ is a strong approximate lattice iff it is an approximate lattice iff it is a lattice by the remark preceding Definition \ref{DefSAL}.

\subsection{Quasi-monotone joining of hulls}\label{SecJoinings}

Throughout this subsection let $G$ be a lcsc group and let $P, Q \in \mathcal C(G)$. We have surjective factor maps $\pi_1: \Omega_{P, Q} \to \Omega_P$ and $\pi_2: \Omega_{P,Q} \to \Omega_Q$. A triple $(\nu_{P, Q}, \nu_P, \nu_Q)$ of probability measures on $\Omega_{P, Q}$, $\Omega_P$ and $\Omega_Q$ respectively is called a \emph{hull joining} if $(\pi_1)_*\nu_{P, Q} = \nu_P$ and $(\pi_2)_*\nu_{P, Q} = \nu_Q$:
\[\begin{xy}\xymatrix{
&(\Omega_{P, Q}, \nu_{P,Q}) \ar[ld]_{\pi_1} \ar[rd]^{\pi_2}&\\
(\Omega_P, \nu_P)&&(\Omega_Q,\nu_Q)
}\end{xy}\]


The hull joining is called \emph{invariant} if $\nu_{P, Q}$ (and hence $\nu_P$ and $\nu_Q$) is $G$-invariant; it is called \emph{$\mu$-stationary} for an admissible probability measure $\mu$ on $G$ if $\nu_{P, Q}$ is $\mu$-stationary, i.e.\ $\mu \ast \nu_{P, Q} = \nu_{P,Q}$. An invariant or $\mu$-stationary hull joining is called \emph{ergodic} if $\nu_{P,Q}$ (and hence $\nu_P$ and $\nu_Q$) is $G$-ergodic.

If $\mu$ is an admissible probability measures on $G$, then by Lemma \ref{MeasureProjection} in Appendix \ref{AppProjection}, every (ergodic) $\mu$-stationary probability measure $\nu_P$ on $\Omega_P$ lifts to an (ergodic) $\mu$-stationary probability measure on $\Omega_{P, Q}$, hence is part of an (ergodic) $\mu$-stationary hull joining $(\nu_{P, Q}, \nu_P, \nu_Q)$. This construction works for arbitrary $P, Q \in \mathcal C(G)$, but there are two caveats:
\begin{itemize}
\item It may happen that $\nu_Q = \delta_\emptyset$, even if $\nu_P$ is non-trivial.
\item Even if $\nu_P$ is $G$-invariant, it will in general not be part of an \emph{invariant} joining (unless $G$ is amenable).
\end{itemize}
To deal with the first point, we need to add some assumption on the relation between $P$ and $Q$. For example, non-triviality of $\nu_Q$ is guaranteed in the case of a \emph{monotone} joining, i.e.\ if $P \subset Q$ and hence $P' \subset Q'$ for all $(P', Q') \in \Omega_{P,Q}$. We will need a slightly more general version of this result.
\begin{lemma}[Quasi-monotone joinings]\label{Joinings}
Let $P,Q \in \cC(G)$ and assume that $P \subset QF$ for some finite set $F \subset G$. Then $P' \subset Q'F$ for all $(P', Q') \in \Omega_{P,Q}$. In particular, if $Q' = \emptyset$, then $P' = \emptyset$.
\end{lemma}
\begin{proof} If $(P', Q') \in \Omega_{P,Q}$, then there exist $g_n \in G$ such that $g_nP \to P'$ and $g_nQ \to Q'$. It thus follows from (CF1) that every $p \in P'$ is the limit of a sequence of the form $(g_np_n)$ with $p_n \in P$. Since $P \subset QF$ we can write $p_n = q_nf_n$ with $q_n \in Q$ and $f_n \in F$. Passing to a subsequence we may assume that $(f_n = f)$ is constant. It then follows from (CF2) that $g_np_nf^{-1}$ converges to an element $q'\in Q'$, and hence $p' = q'f \in Q'F$. 
\end{proof}
\begin{corollary}\label{NonTrivialJoining} Let $P,Q \in \cC(G)$ and let $(\nu_{P, Q}, \nu_P, \nu_Q)$ be a $\mu$-stationary hull joining. If $P \subset QF$ for some finite set $F \subset G$ and if $\nu_P$ is non-trivial, then also $\nu_Q$ is non-trivial.
\end{corollary}
\begin{proof} Otherwise, ${\rm supp}(\nu_{P,Q}) \subset (\Omega_P \times \{\emptyset\}) \cap\Omega_{P,Q}$, and then Lemma \ref{Joinings} would impliy that $\nu_P$ is the Dirac mass at the empty set.
\end{proof}
To deal with the second point, we observe that if $(\nu_{P, Q}, \nu_P, \nu_Q)$ is a $\mu$-stationary hull joining and $\nu_P$ is non-trivial and $G$-invariant, then while $\nu_Q$ need not be invariant, it can at least not satisfy certain strong negations of invariance. To make this precise we recall the following definition.
\begin{definition} Let $H \curvearrowright \Omega$ be a TDS, $\mu$ an admissible probability measure on $H$ and let $\nu$ be a $\mu$-stationary Borel probability measure on $\Omega$. We denote by ${\bf P} := \mu^{\times \mathbb N}$ the product measure on $G^{\times \mathbb N}$.
\begin{enumerate}[(i)]
\item 
Given $\xi = (\xi_n) \in G^{\times \mathbb N}$ we say that $\nu$ has \emph{conditional measure} $\nu_\xi$ with respect to $\xi$ if
\[
(\xi_1 \cdots \xi_n)_*\nu \longrightarrow \nu_\xi
\]
in the weak-$*$-topology as $n \to \infty$.
\item $\nu$ is callled \emph{$\mu$-proximal} if for ${\bf P}$-almost every $\xi \in G^{\times \mathbb N}$ the conditional measure $\mu_\xi$ exists and is a point measure. In this case, $(\Omega, \nu)$ is called a \emph{$\mu$-boundary}.
\end{enumerate}
\end{definition}
In fact, it follows from the martingale convergence theorem that conditional measures exist for  ${\bf P}$-almost every $\xi \in G^{\times \mathbb N}$. Typical examples of $\mu$-boundaries are given by generalized flag varieties: If $k$ is a local field, ${\bf H}$ is a semisimple algebraic group over $k$, and ${\bf P}$ is a parabolic subgroup of ${\bf H}$, then for every admissible probability measure $\mu$ on ${\bf H}(k)$ there exists a unique $\mu$-stationary probability measure $\nu$ on $({\bf H/P})(k)$, and then $(({\bf H/P})(k), \nu)$ is a $\mu$-boundary \cite[Thm. VI.3.7]{Margulis}.
\begin{proposition}\label{NonSAT} Assume that $(\nu_{P, Q}, \nu_P, \nu_Q)$ is an ergodic $\mu$-stationary hull joining and that $P \subset QF$ for some finite set $F \subset G$. If $\nu_P$ is non-trivial and $G$-invariant and $\nu_Q$ is $\mu$-proximal, then there exists $Q' \in \Omega_Q$ such that $G = Q'F$.
\end{proposition}
\begin{proof} Assume for contradiction that $\nu_Q$ was $\mu$-proximal. It would then follow from \cite[Prop. 3.1]{FurstenbergGlasner} that (in the terminology of \emph{loc.\ cit.}) the only $\mu$-joining of $(\Omega_P, \nu_P)$ and $(\Omega_Q, \nu_Q)$ is the product joining. In our terminology this means that $(\Omega_{P \times Q}, \nu_{P, Q}) \cong (\Omega_P \times \Omega_Q, \nu_P \times \nu_Q)$ as measurable $G$-spaces, hence in particular, the support of $\nu_{P, Q}$ would be $G \times G$ invariant. 

Now since $\nu_P$ is non-trivial, we can find $P \in {\rm supp}(\nu_P)$ such that $P' \neq \emptyset$. There then exists $Q' \in {\rm supp}(\nu_Q)$ such that $(P', Q') \in {\rm supp}(\nu_{P,Q})$ and hence $G\times G$-invariance of the latter set implies that $(gP', Q') \in {\rm supp}(\nu_{P,Q})$ for all $g \in G$. By Lemma \ref{Joinings} we thus have $gP' \subset Q'F$ for all $g \in G$ and thus $Q'F = G$.
\end{proof}
\subsection{Commensurability invariance of approximate lattices}
Given a group $G$ we say that two subsets $A, B \subset G$ are \emph{commensurable} if there exist finite subset $F_1, F_2 \subset G$ such that $A \subset BF_1$ and $B \subset AF_2$. Commensurability defines an equivalence relation on subsets of $G$, and as a first application of stationary hull joinings we show that the class of approximate lattices is invariant under commensurability.
\begin{lemma}\label{Commens} Let $\Lambda$ be an approximate lattice in a lcsc group $G$. If $\Lambda' \subset G$ is a discrete approximate subgroup and $\Lambda \subset \Lambda' F$ for some finite subset $F \subset G$,  then also $\Lambda'$ is an approximate lattice. In particular, this is the case if $\Lambda'$ is commensurable to $\Lambda$.
\end{lemma}
\begin{proof} 
Since $\Lambda$ is an approximate lattice, there exists a non-trivial $\mu$-stationary probability measure $\nu_\Lambda$ on $\Omega_\Lambda$ for every admissible probability measure $\mu$ on $G$, and  hence a stationary hull joining $(\mu_{\Lambda, \Lambda'}, \mu_\Lambda, \mu_{\Lambda'})$ between $\Omega_\Lambda$ and $\Omega_{\Lambda'}$. By Corollary \ref{NonTrivialJoining}, the measure $\nu_{\Lambda'}$ is non-trivial, hence $\Lambda'$ is an approximate lattice. 
\end{proof}
If $G$ is amenable, then invariant hull joining always exist, hence we deduce:
\begin{corollary} Let $\Lambda$ be a strong approximate lattice in an amenable lcsc group $G$. If $\Lambda' \subset G$ is a discrete approximate subgroup and $\Lambda \subset \Lambda' F$ for some finite subset $F \subset G$,  then $\Lambda'$ is a strong approximate lattice. In particular, this is the case if $\Lambda'$ is commensurable to $\Lambda$.\qed
\end{corollary}

\section{Algebraic approximate subgroups}\label{SecAlgebraicApproximateGroups}
In this section we are going to study Zariski closures of approximate subgroups of linear algebraic groups.
Throughout this section we fix a field $k$ and a linear algebraic group ${\bf G}$ defined over $k$ and denote by $G:={\bf G}(k)$ its group of $k$-points. We equip $G$ with its Zariski topology and given a subset $A \subset G$ we denote by $\overline{A}^Z$ its Zariski closure in $G$.

\subsection{Zariski closures of approximate subgroups}
We first observe that approximate subgroups are preserved under Zariski closure:
\begin{lemma}[Zariski closures of approximate subgroups]\label{ZariskiClosureAG} If $\Lambda \subset G$ is an approximate subgroup, then its Zariski closure $\overline{\Lambda}^Z \subset G$ is an approximate subgroup as well.
\end{lemma}
Note that if we equip $G$ and $G\times G = ({\bf G \times G})(k)$ with their respective Zariski topologies then inversion and multiplication are continuous. Since the Zariski topology on $G\times G$ is finer than the product topology, $G$ is not a topological group with respect to the Zariski topology. It is however a (non-Hausdorff) \emph{semitopological} group in the sense that multiplication is separately continuous. Thus Lemma \ref{ZariskiClosureAG} is a special case of the following general result:
\begin{proposition} Let $G$ be a semitopological group (not necessarily Hausdorff). Then the closure of an approximate subgroup of $G$ is again an approximate subgroup.
\end{proposition}
\begin{proof} Let $F_\Lambda$ be as in \eqref{FLambda} and let $H$ be the closure of $\Lambda$ in $G$. Enumerate $F_\Lambda = \{g_1, \dots, g_N\}$ and let $\lambda \in \Lambda$.
Then using the fact that left- and right-multiplication by an element of $G$ is a homeomorphism we obtain
\[
\lambda H = \lambda \overline{\Lambda} = \overline{\lambda\Lambda} \subset \overline{\Lambda^2} \subset \overline{\Lambda F_\Lambda} = \overline{\bigcup_{i=1}^N \Lambda g_i} \subset \bigcup_{i=1}^N \overline{\Lambda g_i} = \bigcup_{i=1}^N \overline{\Lambda}g_i = HF_\Lambda
\]
and hence $\Lambda H \subset HF_\Lambda$. Since the right hand side is closed as a finite union of closed set we deduce that for every $h \in H$,
\[
Hh = \overline{\Lambda} h = \overline{\Lambda h} \subset \overline{\Lambda H} \subset HF_\Lambda,
\]
which shows that $H^2 \subset HF_\Lambda$ and finishes the proof.
\end{proof}

\subsection{Algebraic approximate subgroups are almost subgroups}
The following is the main result of this section:
\begin{theorem}[Algebraic approximate subgroups are almost subgroups]\label{AASMain}
Let $k$ be a field and ${\bf G}$ be a linear algebraic group over $k$. If $\Lambda \subset {\bf G}(k)$ is an approximate subgroup, then there exists a $k$-algebraic subgroup ${\bf H}$ of ${\bf G}$, an element $g \in {\bf G}(k)$ and a finite subset $F \subset {\bf G}(k)$ such that \[g{\bf H}(k) \subset \overline{\Lambda}^Z \subset F {\bf H}(k) \cap {\bf H}(k)F.\] 
\end{theorem}
For the proof we recall that a topological space is called \emph{irreducible} if it is non-empty and it is not the union of two proper closed subsets. We call a (not necessarily Zariski-closed) subset of $G:={\bf G}(k)$ irreducible, if it is irreducible with respect to the restriction of the Zariski topology from $G$. We will need the
following lemma, which will be proved in the next subsection.
\begin{lemma}\label{ProductOfIrreducible} Let $A, B \subset G$ be irreducible subsets. Then ${A^{-1}B}$ is irreducible as well.
\end{lemma}
\begin{proof}[Proof of Theorem \ref{AASMain}]  In view of Lemma \ref{ZariskiClosureAG} we can replace $\Lambda$ by its Zariski closure and thereby assume that $\Lambda$ is Zariski closed. Let $F_\Lambda = \{g_1, \dots, g_N\}$ be as in \eqref{FLambda} and let $\Lambda_0, \dots, \Lambda_m$ be the irreducible components of $\Lambda$, where $\Lambda_0$ is an irreducible component of maximal dimension. We claim that if $\lambda_0 \in \Lambda_0$, then $H := \lambda_0^{-1}\Lambda_0$ is a subgroup of ${\bf G}(k)$. Since $e \in H$ we have $H \subset H^{-1}H$, and it remains to show $H^{-1}H \subset H$.  We have
\[
H^{-1}H = \Lambda_0^{-1}\Lambda_0 \subset \Lambda^2 \subset \bigcup_{i=1}^N g_i \Lambda = \bigcup_{i=1}^N\bigcup_{j=0}^m g_i\Lambda_j.
\]
Now $H$ is irreducible, since $\Lambda_0$ is irreducible, and hence also ${H^{-1}H}$ is irreducible by Lemma \ref{ProductOfIrreducible}. There thus exists  $i \in \{1, \dots, n\}$ and an irreducible component $\Lambda_j$ of $\Lambda$ such that 
\begin{equation}
g_i\Lambda_j \supset H^{-1}H \supset H,
\end{equation}
where we have used that $e \in H^{-1}$. Since $g_i\Lambda_j$ is Zariski closed we also have $\overline{H^{-1}H}^Z \subset g_i\Lambda_j$. Since $\Lambda_0$ was an irreducible component of maximal dimension, we have
\[
\dim g_i\Lambda_j = \dim \Lambda_j \leq \dim \Lambda_0 = \dim H \leq \dim \overline{H^{-1}H}^Z \leq \dim g_i\Lambda_j.
\]
We deduce that $\dim H = \dim \overline{H^{-1}H}^Z$, and since $H \subset H^{-1}H$ and the latter is irreducible we have $H = H^{-1}H$. Thus $H$ is a Zariski closed subgroup of $G$, and hence 
$H = {\bf H}(k)$, where ${\bf H}$ is the algebraic subgroup of ${\bf G}$ defined by the vanishing ideal of $H$.

It remains only to show that finitely many left-cosets of $H$, or equivalently $\Lambda_0$, cover $\Lambda$, since then also finitely many right cosets of $H$ cover $\Lambda$ by symmetry of $H$ and $\Lambda$. Assume otherwise; then the set $\{\lambda \Lambda_0 \mid \lambda \in \Lambda\}$ would be infinite, hence we could find a sequence $(\lambda_n)$ in $\Lambda$ such that $\lambda_i \Lambda_0 \neq \lambda_j \Lambda_0$ for all $i \neq j$. We then have
\[
\bigsqcup_{n=1}^\infty \lambda_n\Lambda_0 \subset \Lambda^2 \subset \bigcup_{i=1}^N g_i \Lambda \subset \bigcup_{i=1}^N \bigcup_{j=0}^m g_i\Lambda_j
\]
By the pigeonhole principle, one of the irreducible sets $g_i\Lambda_j$ would thus be a disjoint union of infinitely many irreducible subsets of the form $\lambda_n\Lambda_0$. Since $\dim \lambda_n\Lambda_0 \geq \dim g_i\Lambda_j$, this is a contradiction.
\end{proof}

\subsection{Product sets of irreducible sets are irreducible}
This subsection is devoted to the proof of Lemma \ref{ProductOfIrreducible}. We keep the notation of the previous subsection. In particular, $k$ denotes a field, ${\bf G}$ a linear algebraic group defined over $k$ and $G:={\bf G}(k)$. 

To show Lemma \ref{ProductOfIrreducible} we first observe that the image of an irreducible topological space under a continuous map is irreducible, and that the map  $q: G \times G \to G$ given by $q(a,b) := a^{-1}b$ is continuous, if $G$ and $G\times G = ({\bf G} \times{\bf G})(k)$ are equipped with their respective Zariski topology. Since $A^{-1}B = q(A \times B)$ it thus suffices to show that if $A$ and $B$ are irreducible subsets of $G$, then $A\times B$ is irreducible in $G \times G$.  We can choose a representation $\rho: {\bf G} \to {\bf GL}_n$ defined over $k$ and thereby consider $G$ as a subset of $k^{n^2}$. It then suffices to establish the following (see \cite[Exercise~3.15 (a)]{Hartshorne}):
\begin{lemma}\label{Lem:product of irred is irred}
Let $ X $ be an irreducible subset of $ k^n $, and let $ Y $ be an irreducible subset of $ k^m $. Then $ X\times Y $ is irreducible in $ k^{n+m} $.
\end{lemma}
\begin{proof}
Let $ Z_1 $ and $ Z_2 $ be closed subsets of $ k^{n+m} $, with corresponding vanishing ideals $ I_{Z_1} $ and $ I_{Z_2} $ in $ k[T_1,\dots ,T_{n+m}] $. Assume that $ X\times Y $ is contained in $ Z_1\cup Z_2 $. We have to show that either $ X\times Y\subset Z_1 $ or $ X\times Y\subset Z_2 $. For $ i\in \lbrace 1,2\rbrace $ we denote $ X_i = \lbrace x\in X~\vert~\lbrace x\rbrace\times Y\subset Z_i \rbrace $.

We first claim that  $ X\subset X_1\cup X_2 $. Indeed, $ \lbrace x\rbrace\times Y$ is the image of $Y$ under the Zariski continuous map $ k^m\to k^{n+m}$ given by $ b\mapsto (x,b) $, hence
an irreducible subset of $ k^{n+m} $, and thus of $ Z_1\cup Z_2 $. We deduce that for all $ x\in X $ we have either $ \lbrace x\rbrace\times Y\subset Z_1 $ or $ \lbrace x\rbrace\times Y\subset Z_2 $, which proves the first claim.

Secondly, we claim that $X_1$ is closed in $k^n$. Given $ b\in k^m $ and $ f\in k[T_1,\dots ,T_{n+m}] $, we define $ f_b\in k[T_1,\dots ,T_{n}] $ by $f_b(T_1, \dots, T_n) :=  f(T_1,\dots ,T_n,b_1,\dots ,b_m) $. Now for every $ x\in X $ we have $ \lbrace x\rbrace\times Y\subset Z_1 $ if and only if $ f(x,y) = 0 $ for all $ f\in I_{Z_1} $ and all $ y\in Y $, and hence
\[
X_1 = \lbrace x\in X~\vert~f_y(x) = 0 \text{ for all } y\in Y \text{  and for all } f\in I_{Z_1}\rbrace
\]
is closed. This proves the second claim, and the same argument shows that $X_2$ is closed in $k^n$.

We have written $X = X_1 \cup X_2$ as the union of two proper closed subsets. Since $X$ is irreducible this implies that either $ X\subset X_1 $ or $ X\subset X_2 $. Consequently we have either $ X\times Y\subset Z_1 $ or $ X\times Y\subset Z_2 $, which finishes the proof.
\end{proof}
This finishes the proof of Lemma \ref{ProductOfIrreducible}.

\section{Proof of Borel density}\label{SecMain}
\subsection{General setting}\label{SecSetting}

Throughout this section $k$ denotes a local field,  ${\bf G}$ is a connected linear algebraic group defined over $k$ and $\Lambda$ denotes an approximate subgroup of
 $G := {\bf G}(k)$. By Theorem \ref{AASMain} there exists an algebraic subgroup ${\bf H}$ of ${\bf G}$, an element $g \in G$ and a finite subset $F \subset G$ such that 
\[
g{\bf H}(k) \subset \overline{\Lambda}^Z \subset F{\bf H}(k) \cap {\bf H}(k)F.
\]
We will abbreviate $H := {\bf H}(k)$ enumerate $F = \{g_1, \dots, g_N\}$ so that
\begin{equation}\label{LambdaH}
\Lambda \subset \overline{\Lambda}^Z \subset \bigcup_{j=1}^n g_jH \qand \Lambda \subset \overline{\Lambda}^Z \subset \bigcup_{j=1}^n Hg_j
\end{equation}
We then have to show that $H = G$. The argument for this will be different in the uniform case (where we use relative denseness of $\Lambda$) and in the strong case (where we use a joining argument between the hulls of $\Lambda$ and $H$).

 
\subsection{The uniform case}
We consider the general setting (and notation) of Subsection \ref{SecSetting}. In addition we are going to assume that $\Lambda \subset G$ is a \emph{uniform} approximate lattice. From this assumption and \eqref{LambdaH} one immediately deduces:
\begin{lemma}\label{Hcocompact} If $\Lambda$ is a uniform approximate lattice, then $H$ is cocompact in $G$.
\end{lemma}
\begin{proof} Since $\Lambda$ is relatively dense, so is its superset $g_1H \cup \dots \cup g_nH$. It thus follows from Lemma \ref{lemma_unformindex} that there exists $j \in \{1, \dots, n\}$ such that $(g_jH)^{-1}g_jH = H^{-1}H = H$ is relatively dense, i.e.\ cocompact.
\end{proof}
There is also a more subtle consequence of \eqref{LambdaH} and the assumption that $\Lambda$ be uniform. Namely, let $\Lambda_j := \Lambda \cap g_jH$ so that $
\Lambda = \Lambda_1 \cup \dots \cup \Lambda_N$. Since $\Lambda$ is relatively dense in $G$, Lemma \ref{lemma_unformindex} implies that there exists $j \in \{1, \dots, N\}$ such that $\Delta_j := \Lambda_j^{-1}\Lambda_j$ is relatively dense in $G$. Note that
\[
\Delta_j = (\Lambda \cap g_jH)^{-1}(\Lambda \cap gH_j) \subset \Lambda^2 \cap H \subset \Lambda^2.
\]
In particular, $\Delta_j$ is a symmetric subset of the uniform approximate lattice $\Lambda^2$ which contains the identity and is relatively dense in $G$. It thus follows from \cite[Corollary 2.10]{BH1}, that $\Delta_j$ is a uniform approximate lattice in $G$ itself. Since $\Delta_j \subset H$ we deduce:
\begin{lemma}\label{HUAL} It $\Lambda$ is a uniform approximate lattice, then $H$ contains a uniform approximate lattice.
\end{lemma}
In \cite[Thm. 5.8]{BH1} it was established that if a compactly generated lcsc group contains a uniform approximate lattice, then it is unimodular. In Theorem \ref{Unimodular} in 
Appendix \ref{AppUnimodular} we will show that this also holds without the assumption of compact generation. This then implies that $H$ is unimodular, in addition to being cocompact. We have established:
\begin{theorem}\label{UniformCase} Let $k$ be a local field and ${\bf G}$ be a connected affine algebraic group defined over $k$. Assume that ${\bf G}$ does not contain any proper algebraic subgroup ${\bf H}$ such that ${\bf H}(k)$ is unimodular and cocompact in ${\bf G}(k)$. Then every uniform approximate lattice in ${\bf G}(k)$ is Zariski dense. 
\end{theorem}
\begin{proof} In view of Lemma \ref{Hcocompact}, Lemma \ref{HUAL} and Theorem \ref{Unimodular}, the assumption forces $H=G$, hence $\Lambda$ is Zariski dense. 
\end{proof}
\begin{corollary}[Borel density, uniform case]  Let $k$ be a local field, ${\bf G}$ be a connected semisimple algebraic group and assume that $G := {\bf G}(k)$ has no compact factors. Then every uniform approximate lattice in $G$ is Zariski dense.
\end{corollary}
\begin{proof} It only remains to check that the assumptions of Theorem \ref{UniformCase} are satisfied in this case. This follows from \cite[Corollary 2.3]{St}.
\end{proof}

\subsection{The non-uniform case}
We consider the general setting (and notation) of Subsection \ref{SecSetting}, but assume in addition that $\Lambda \subset G$ is a \emph{strong} approximate lattice. We also fix a left-Haar measure $m_G$ and an admissible probability measure $\mu = \rho m_G$ on $G$. We assume that $\mu$ is symmetric and that $\rho \in L^1(G) \cap L^\infty(G)$. From the results of Subsection \ref{SecJoinings} we then deduce:
\begin{lemma}\label{Existsnu}
If $H \neq G$, then there exists a $\mu$-stationary ergodic probability measure $\nu$ on $G/H$ which is not $\mu$-proximal.
\end{lemma}
\begin{proof} Since $\Lambda$ is a strong approximate lattice, we can choose a non-trivial $G$-invariant ergodic probability measure on $\Omega_\Lambda$. This measure then induces via hull joining a $\mu$-stationary ergodic probability measure $\nu$ on $\Omega_H$. Since $\Lambda \subset HF$ we deduce from Lemma \ref{NonTrivialJoining} that $\nu$ is non-trivial, hence it is supported on the orbit $G/H \subset \Omega_H$ by Lemma \ref{lemma2}. To see that $\nu$ is not $\mu$-proximal it suffices to show by Lemma \ref{NonSAT} that if $H' \in \Omega_H$, then $H'F \subsetneq G$. Since ${\bf G}$ is connected and ${\bf H}$ is a proper algebraic subgroup of ${\bf G}$ we have $\dim{\bf H} < \dim{\bf G}$, and hence $H$ has infinite index in $G$. Since $\Omega_H \subset G/H \cup\{\emptyset\}$, the lemma follows.
\end{proof}
We thus have to investigate, which homogeneous spaces of the form $G/H$ admit non-proximal stationary probability measures. Note that since $G$ contains a strong approximate lattice, it is automatically unimodular by \cite[Thm. 5.8]{BH1}.
\begin{lemma}
\label{lemma3}
Assume that $G/H$ admits a $\mu$-stationary probability measure $\nu$. Then $\nu$ is actually the unique $\mu$-stationary probability measure on $G/H$, and if $H$ is unimodular, then $\nu$ is $G$-invariant, and hence $H$ has finite covolume.
\end{lemma}
\begin{proof} Every $\mu$-stationary probability measure on $G/H$ is $G$-quasi-invariant. If there was more than one $\mu$-stationary $\mu$-probability measure on $G/H$, then there would be two different ergodic such measures, and these would then be mutually singular. This would contradict the fact that the quotient $G/H$ admits a unique $G$-invariant measure class.

Now assume that $H$ is unimodular. We are going to show that $H$ has finite covolume. This will finish the proof, since then $\nu$ must be the unique invariant probability measure by the uniqueness statement. Assume for contradiction that $H$ has infinite covolume and denote by $\eta$ the infinite $G$-invariant Radon measure on $G/H$. Since $\nu$ and $\eta$ both represent the unique $G$-invariant measure class on $G/H$ we can write $\nu = u \eta$ for some non-negative $\eta$-integrable Borel function $u$ on $G/H$. Since $\nu$ is $\mu$-stationary and  $\mu = \rho m_G$ is symmetric we deduce that
\[\rho * u  =  u.\] 
Moreover, since $\rho \in L^\infty(G)$ we have, by H\"older's inequality,
\[
\|u\|_\infty = \|\rho * u\|_\infty \leq \|\rho\|_\infty \|u\|_1 < \infty.
\]
In particular, $u \in L^1(\eta) \cap L^\infty(\eta) \subset L^2(\eta)$ and $u$ is continuous. Since $\mu * u  = \check \mu \ast u = u$  and $\eta$ is  $G$-invariant we have 
\begin{eqnarray*}
&&\int_{G/H} \int_G \big| u(gx) - u(x)\big|^2 \, d\mu(g) \, d\eta(x) \\&=& \int_{G/H} \int_G u(gx)^2  d\mu(g) \, d\eta(x) -2 \int_{G/H} \int_G u(gx)d\mu(g)\,u(x) \, d\eta(x) + \int_{G/H} u(x)^2 d\eta(x)\\
&=& \int_G  \int_{G/H}u(gx)^2  d\eta(x)\,  d\mu(g) - 2  \int_{G/H} (\rho \ast u)(x) u(x)d\eta(x) + \int_{G/H} u(x)^2 d\eta(x)\\
&=& \int_G  \int_{G/H}u(x)^2  d\eta(x)\,  d\mu(g) - 2  \int_{G/H} u(x) u(x)d\eta(x) + \int_{G/H} u(x)^2 d\eta(x) \quad = \quad 0.
\end{eqnarray*}
We conclude that $u$ (as an element in $L^2$) is invariant under $\mu$-a.e. $g \in G$. Since the support of $\mu$ generates $G$, $u$ is $G$-invariant and thus
constant, whence not $\eta$-integrable. This contradicts the finiteness of $\nu$.
\end{proof}
\begin{corollary}\label{AbstractBorelDensity} Let $k$ be a local field, ${\bf G}$ be a connected algebraic group over $k$ and assume that every Zariski closed proper subgroup of ${\bf G}(k)$ is contained in a closed subgroup $M<G$ satisfying one of the following three properties:
\begin{enumerate}[(i)]
\item $M$ is unimodular of infinite covolume in $G$.
\item $G/M$ admits a unique $\mu$-stationary measure which is $\mu$-proximal.
\item $G/M$ does not admit a $\mu$-stationary probability measure.
\end{enumerate}
Then every strong approximate lattice in ${\bf G}(k)$ is Zariski dense.
\end{corollary}
\begin{proof} In that notation of Subsection \ref{SecSetting} we have to show that $H = G$. Assume otherwise, and let $M < G$ be a subgroup containing $H$ as in the corollary.
By Lemma \ref{Existsnu} there exists a $\mu$-stationary measure $\nu$ on $G/H$, and we denote by $\nu_1$ its push-forward to $G/M$. Since $\nu$ is $\mu$-stationary, but not $\mu$-proximal the same holds for $\nu_1$, and hence (ii) and (iii) cannot hold. This forces $M$ to be unimodular of infinite covolume, which contradicts Lemma \ref{lemma3}.
\end{proof}
We conclude:
\begin{theorem}[Borel density for strong approximate lattices]\label{StrongCase} Let $k$ be a local field and let ${\bf G}$ be a connected semisimple algebraic group over $k$. If $G := {\bf G}(k)$ does not have any compact factors, then every strong approximate lattice $\Lambda \subset G$ is Zariski dense.
\end{theorem}
\begin{proof} We have to check that the conditions of Corollary \ref{AbstractBorelDensity} are satisfied. By \cite[Lemma 2.2]{St} every Zariski closed subgroup of $G$ is contained in either a parabolic subgroub $P$ of $G$ or in an algebraic subgroup $M$ whose identity component $M^0$ is reductive with anisotropic center over $k$. In the first case, $G/P$ admits a unique $\mu$-stationary measure which is $\mu$-proximal \cite[Thm. VI.3.7]{Margulis}. In the second case, $M$ is unimodular, since $M^0$ is reductive and $[M:M^0]$ is finite. In this case, it then has infinite covolume in $G$ by \cite[Corollary 2.3]{St}. Now the theorem follows from Corollary \ref{AbstractBorelDensity}.
\end{proof}

\section{Variants and refinements}\label{SecVariants}

\subsection{Dani--Shalom density}\label{SecDS}
Let ${\bf G}$ be a connected affine algebraic group defined over a local field $k$. In this subsection we will assume that $G = {\bf G}(k)$ is amenable. In particular, this is the case if ${\bf G}$ is solvable. 
\begin{theorem}[Dani--Shalom density for strong approximate lattices]\label{BDS} Assume that $G$ does not contain any proper normal cocompact algebraic subgroup. Then every strong approximate lattice and every uniform approximate lattice in $G$ is Zariski dense.
\end{theorem}
\begin{proof} Since $G$ is amenable, every uniform approximate lattice in $G$ is strong, hence we assume that $\Lambda \subset G$ is a strong approximate lattice. We then define $H$ as in Subsection \ref{SecSetting}. There then exists a finite set $F$ such that $\Lambda \subset HF$. Since $\Lambda$ is a strong approximate lattice, there exists a $G$-invariant ergodic probability measure $\nu_\Lambda$ on $\Omega_\Lambda$. Since $G$ is amenable, this measure is part of a $G$-invariant ergodic hull joining $(\nu_{\Lambda, H}, \nu_\Lambda, \nu_H)$ between $\Omega_\Lambda$ and $\Omega_H$, and the invariant probability measure $\nu_H$ is non-trivial by Lemma \ref{NonTrivialJoining}, hence supported on the orbit $G/H \subset \Omega_H$ by Lemma \ref{lemma2}. In particular, $\nu_H$ has full support on $G/H$.

Now ${\rm supp}(\nu_H) = G/H$ is a subset of $({\bf G/H})(k)$, and since ${\bf G}$ acts algebraically on ${\bf G/H}$, it follows from \cite[Thm. 1.1]{Shalom} (which generalizes \cite[Cor. 2.6]{Dani}) that this support consists entirely of $G$-fixpoints. This forces $H = G$ and finishes the proof.
\end{proof}
\begin{remark}\label{RemkSplit} Solvable algebraic groups over local fields $k$ whose $k$ points do not admit any proper normal algebraic cocompact subgroups are called \emph{$k$-discompact} and have been characterized by Shalom in \cite[Thm. 3.6]{Shalom}. They are precisely the $k$-algebraic solvable group which are \emph{$k$-split} in the sense that every composition factor is isomorphic to either the additive or multiplicative group over $k$.
\end{remark}
Taken together, Theorem \ref{UniformCase}, Theorem \ref{StrongCase}, Theorem \ref{BDS} and Remark \ref{RemkSplit} now yield our Main Theorem from the introduction.

\subsection{The case of unipotent groups}\label{SecUnipotent}

Theorem \ref{BDS} applies in particular to unipotent algebraic groups. Over local fields $k$ of positive characteristic there do exist non-split unipotent groups.
\begin{example}[Rosenlicht]
Let $p$ be a prime, $k:= \mathbb F_p((t))$ and consider the algebraic subgroup of the additive group of $k^2$ whose $k$-points are given by\[
\{(x,y) \in k^2 \mid y^p = tx^p-x\}.
\]
One checks that its group of $k$-points is actually infinite and contained in $(F_p[[t]])^2$, hence it is an example of a \emph{compact} unipotent group. Every finite subset of this group is thus a uniform approximate lattice, which is not Zariski dense.
\end{example}
The natural context of this example is that of $k$-wound unipotent groups. A unipotent algebraic group ${\bf G}$ defined over a field $k$ is called \emph{$k$-wound} if every $k$-morphism from the additive group of $k$ to ${\bf G}$ is constant. Over a local field this is equivalent to compactness of ${\bf G}(k)$ \cite[Sec. VI, § 1, Th\'eor\`eme]{Oes84}. If $k$ is of characteristic $0$, then every $k$-wound unipotent groups is trivial \cite[Cor. 14.55]{Mil17}, but Rosenlicht's example shows that non-trivial examples exist in positive characteristic. In general, if ${\bf G}$ is any unipotent algebraic group over an arbitrary field $k$, then there exists a unique $k$-split unipotent normal subgroup ${\bf G}_s$ of ${\bf G}$, such that the quotient group ${\bf G}/{\bf G}_s$ is $k$-wound \cite[p.\ 733, § 4.2, Theorem]{Tit13}. In characteristic $0$ we thus have ${\bf G} = {\bf G}_s$, hence we deduce:
\begin{example} Let $k$ be a local field of characteristic $0$ and let ${\bf G}$ be a unipotent algebraic group over $k$. Then every strong approximate lattice and every uniform approximate lattice in ${\bf G}(k)$ is Zariski-dense.
\end{example}
We do not know whether Theorem \ref{BDS} holds for approximate lattices which are not strong. However, the following example shows that it does not hold for \emph{weak approximate lattices}. Here, a discrete approximate subgroup $\Lambda$ of a lcsc group $G$ is called a \emph{weak approximate lattice} if its hull admits a non-trivial $\mu$-stationary probability measure for \emph{some} admissible probability measure $\mu$ on $G$. 
\begin{example}
Consider the algebraic group ${\bf G} := {\bf GL}_1 \ltimes {\mathbb A}^1$ over $\bR$ so that $G = {\bf G}(\bR) = \bR^\times \ltimes \bR$ is the $(ax+b)$-group. By \cite[Sec. 5.4]{BH1} the subgroup $\Lambda := \{1\} \rtimes \bZ$ is a weak approximate lattice in $G$, but its Zariski closure is given by $\{1\} \rtimes \bR$. Thus weak approximate lattices in the $(ax+b)$-group need not be Zariski dense, despite the fact that $G$ is solvable without cocompact algebraic subgroups.
\end{example}


\subsection{Thin approximate subgroups of abelian groups}
If ${\bf G}$ is an algebraic group over a local field $k$, then an approximate subgroup $\Lambda \subset G$ is called \emph{thin} if it is Zariski dense, but not an approximate lattice. In the case where $\Lambda$ is an actual subgroup one recovers the notion of a thin subgroup. It is well-known that nilpotent algebraic groups over $\mathbb R$ do not admit thin subgroups. On the contrary we show:
\begin{proposition} The additive group $\bR^2$ admits thin approximate subgroups.
\end{proposition}
\begin{proof} We set
\[
\Gamma = \big\{ (m+n\sqrt{2},m-n\sqrt{2}) \mid m,n \in \bZ\} 
\qand
S = \bR \times [-1,1].
\] 
It is easy to check that $\Lambda := \Gamma \cap S$ is an infinite approximate subgroup, but not a uniform approximate lattice, hence not an approximate lattice at all by \cite[Cor. 4.19]{BH1}. It remains to show that $\Lambda$ is Zariski-dense. Otherwise, by Theorem \ref{AASMain}, $\Lambda$ would be contained in a finite union of translates of a fixed proper algebraic subgroup of $G$. Since $\Lambda$ is infinite, $H$ would have to be non-trivial, hence a line. This implies that either $H \subset S$ or that $H \cap S$ is compact.
In the second case, $S \cap \Lambda$ would have to be contained in a compact subset of $G$; since $\Gamma$ is discrete, this implies that 
$\Lambda$ is finite, a contradiction. Thus $H \subset S$, and thus all points of $\Lambda$ lie on a finite union of lines which are parallel to the line $\bR \times \{0\}$. Then there exist 
 $\alpha_1,\ldots,\alpha_N \in [-1,1]$ such that
 \[
 \Lambda \subset \bigcup_{k=1}^N \big\{ (m+n\sqrt{2},m-n\sqrt{2}) \mid m-n\sqrt{2} = \alpha_k \big\}
 \]
 Thus the second coordinate of elements of $\Lambda$ can take only finitely many values, but since the first coordinate is just the Galois conjugate of the second coordinate we deduce that $\Lambda$ is actually finite, which is a contradiction.
\end{proof}

\appendix

\section{Lifting stationary measures}\label{AppProjection}
The purpose of this appendix is to record a proof of the following fact from measure theory, to be used in Subsection \ref{SecJoinings}. Given a lcsc group $G$ and a compact $G$-space $\Omega$ and an admissible probability measure $\mu$ on $G$, we denote by $ {\rm Prob}_\mu(\Omega) \subset {\rm Prob}(\Omega)$ the compact convex sets of $\mu$-stationary, respectively arbitrary probability measures on $\Omega$. 
\begin{lemma}\label{MeasureProjection}
Let $X$ and $Y$ be compact $G$-spaces, and suppose that there exists a continuous surjective $G$-map $\pi : X \ra Y$. Then the induced map
$\pi_* : {\rm Prob}_\mu(X) \ra {\rm Prob}_\mu(Y)$ is surjective as well, and maps ergodic measures surjectively onto ergodic measures. \qed
\end{lemma}
\begin{proof} We first show surjectivity of the map $\pi_*: {\rm Prob}(X) \ra {\rm Prob}(Y)$. Thus let $\eta \in {\rm Prob}(Y)$ and define $\eta': \pi^*(C(Y)) \to \mathbb R$ by $\eta'(\pi^*(f)) := \eta(f)$, which is well-defined since $\pi$ is surjective and thus $\pi^*:C(Y) \to C(X)$ is injective. We have $\|\eta'\| \leq \|\eta\| = 1$ and since $\eta'(1) = 1$ we deduce that $\|\eta'\| = 1$. By Hahn-Banach we can thus extend $\eta'$ to a continuous linear functional $\eta''$ of norm $1$ on all of $C(X)$, and by construction $\pi_*\eta'' = \eta$. It thus remains to show only that $\eta''$ is a positive linear functional on $C(X)$. Thus let $f \in C(X)$ be non-negative so that $\|f\|_\infty \geq \|f\|_\infty - f \geq 0$. Since $\|\eta''\|=1$ and $\eta''(1) = 1$ we have 
\[
\|f\|_\infty \geq \|f\|_\infty - f \geq | \eta''(\|f\|_\infty -f)| \geq \eta''(\|f\|_\infty -f) = \|f\|_\infty - \eta''(f),
\]
and hence $\eta''(f) \geq 0$. This shows that $\eta''$ is positive, and hence $\pi_*: {\rm Prob}(X) \ra {\rm Prob}(Y)$ is surjective. Given $\eta \in {\rm Prob}(Y)$ we now defined
weak-$*$-compact convex sets by
\[
F(\eta) := \{\nu \in {\rm Prob}(X) \mid \pi_*\nu = \eta\} \qand F_\mu(\eta) := \{\nu \in {\rm Prob}_\mu(X) \mid \pi_*\nu = \eta\}.
\]
We have just seen that $F(\eta)$ is non-empty for every $\eta \in {\rm Prob}(Y)$, and if $\eta$ is moreover $\mu$-stationary, then it is invariant under convolution by $\mu$, since $\pi$ is $G$-equivariant and thus for all $\nu \in F(\eta)$ we have
\[
\pi_*(\mu*\nu) = \mu\ast (\pi_*\nu) = \mu \ast \eta = \eta.
\]
It then follows from the Markov-Kakutani fixpoint theorem that $F_\mu(\eta)$ is also non-empty. This shows that $\pi_* : {\rm Prob}_\mu(X) \ra {\rm Prob}_\mu(Y)$ is surjective.

For the second statement we first recall from \cite[Cor. 2.7]{BaderShalom} that the ergodic $\mu$-stationary probability measures are precisely the extremal points of the convex compact set of all $\mu$-stationary probability measures. Assume now that $\eta \in {\rm Prob}_\mu(Y)$ is ergodic and let $\nu$ be an extremal point of $F_\mu(\eta)$, which exists by 
the Krein-Milman theorem since $F_\mu(\eta) \neq \emptyset$. We claim that $\nu$ is ergodic, i.e. an extremal point of ${\rm Prob}_\mu(X)$. Otherwise we could write $\nu$ as $\nu = \alpha_1\nu_1+ \alpha_2 \nu_2$ for some $\alpha_1, \alpha_2 \in (0, 1)$ with $\alpha_1 + \alpha_ 2 = 1$ and $\nu_1, \nu_2 \in {\rm Prob}_\mu(X)$. But then
\[
\eta = \pi_*\nu = \alpha_1 \pi_*\nu_1 + \alpha_2 \pi_*\nu_2,
\]
and hence ergodicity of $\eta$ forces $\pi_*\nu_1 = \pi_*\nu_2 = \eta$ and hence $\nu_1, \nu_2 \in F_\mu(\eta)$. This contradicts extremality of $\nu$ in $F_\mu(\eta)$, and hence $\nu$ must have been ergodic.
\end{proof}

\section{The unimodularity theorem revisited}\label{AppUnimodular}
The following theorem is used in the proof of Borel density in the uniform case:
\begin{theorem}[Unimodularity theorem, refined version]\label{Unimodular}
Let $G$ be a lcsc group which contains a uniform approximate lattice $\Lambda$. Then $G$ is unimodular.
\end{theorem}
Under the additional assumption that $G$ be \emph{compactly generated}, this theorem was established in \cite[Thm. 5.8]{BH1}. We revisit the proof to establish the above stronger version; we use this opportunity to correct a few inequalities in the original proof. The following lemma replaces \cite[Lemma 5.10]{BH1}. Here, $m_G$ denotes a fixed choice of left-Haar measure on $G$ and $\Delta_G$ denotes the modular function of $G$.
\begin{lemma}\label{LemmaUniNew} Assume that $G$ is a non-unimodular lcsc group. Then there exists $\rho \in C(G)$ with the following properties:
\begin{enumerate}[(i)]
\item $\rho(g) > 0$ for all $g\in G$ (hence in particular ${\rm supp}(\rho) = G$).
\item  $\int_G \rho(t)\, dm_G(t) = 1$.
\item $\int_G \rho(t) \Delta_G(t)\, dm_G(t) > 1$.
\end{enumerate}
\end{lemma}
We will use following simple observation, which we leave as an exercise:
\begin{lemma}\label{SequenceLemma} Let $S$ be a countable set. Then for every function $b: S \to [0, \infty)$ there exists $a:S \to (0, \infty)$ such that
\[\pushQED{\qed}
\sum_{s \in S} a(s) = 1 \qand \sum_{s \in S} a(s)b(s) < \infty.\qedhere
\popQED
\]
\end{lemma}
\begin{proof}[Proof of Lemma \ref{LemmaUniNew}] We first construct a function $\rho_0 \in C(G)$ which satisfies (i), (ii) and
\begin{enumerate}[(iii')]
\item $\gamma := \int_G \rho_0(s) \Delta_G(s)\, dm_G(s) <\infty$.
\end{enumerate}
To construct $\rho_0$ we pick a countable dense subset $S \subset G$ (which exists since $G$ is second countable) and define $b:S \to [0, \infty)$ by $b(s) := \Delta_G(s)^{-1}$. Using Lemma \ref{SequenceLemma} we then choose a function $a: S \to [0, \infty)$ such that
\begin{equation}\label{aDelta}
\sum_{s \in S}a(s) = 1 \qand \sum_{s \in S} a(s)\Delta_G(s)^{-1} < \infty.
\end{equation}
Next we pick a compactly supported function $\varphi \in C_c(G)$ with $\varphi \geq 0$ and $\|\varphi\|_1 = 1$ and set
\[
\rho_0(t) := \sum_{s \in S} a(s)\varphi(st).
\]
We check that $\rho_0(t)$ satisfies (i), (ii) and (iii'): Property (i) is immediate, since $S$ is dense in $G$ and $\varphi$ is positive on a non-empty open set, and (ii) follows from
\[
\int_G \rho_0(t)\, dm_G(t) =  \sum_{s \in S} a(s) \int_G \varphi(st)\, dm_G(t) =  \sum_{s \in S} a(s) \int_G \varphi(t)\, dm_G(t) = \sum_{s \in S} a(s) = 1.
\]
Finally, (iii') follows from the fact that
\[
 \int_G \rho(t) \Delta_G(t)\, dm_G(t) = \sum_{s \in S}a(s)  \int_G\varphi(st) \Delta_G(t)\, dm_G(t) = \sum_{s \in S} a(s) \Delta(s)^{-1} \int_G\varphi(t) \Delta_G(t) dm_G(t). 
\] 
Since $\varphi$ is compactly supported, the integral $\int_G\varphi(t) \Delta_G(t) dm_G(t)$ converges, and thus the sum is finite by \eqref{aDelta}. We have thus constructed $\rho_0$ satisfying (i), (ii) and (iii').

Now we choose $a >0$ such that $a\gamma > 1/2$; since $\Delta_G$ is unbounded we then find $s\in G$ such that $(1-a)\gamma\Delta(s)^{-1} > 1/2$. We claim that 
\[
\rho(t) := a\rho_0(t) + (1-a)\rho_0(st)
\]
satisfies (i)-(iii). Here, (i) and (ii) are immediate from the corresponding properties of $\rho_0$ and left-invariance of $m_G$. Concerning (iii) we observe that, using left-invariance of $m_G$ and the fact that $\Delta_G$ is a homomorphism,
\begin{eqnarray*}
\int_G \rho(t)\delta(t) \, dm_G(t) &=& a\gamma + (1-a) \int_G \rho_0(t)\Delta_G(s^{-1}t) dm_G(t)\\
&=& a\gamma + (1-a)\gamma \Delta_G(s)^{-1} \quad > \quad 1/2+1/2 \quad = \quad 1.
\end{eqnarray*}
This establishes (iii) and finishes the proof.
\end{proof}
Towards the proof of Theorem \ref{Unimodular} we now assume for contradiction that $G$ contains a uniform approximate lattice $\Lambda$, but is non-unimodular. We then choose  $\rho$ as in Lemma \ref{LemmaUniNew} and define an admissible probability measure $\mu$ on $G$ by
\[\mu := \rho m_G.\] 
By the Markov--Kakutani fixpoint theorem there then exists a $\mu$-stationary probability measure $\nu$ on $\Omega_\Lambda$. Since $\Lambda$ is relatively dense we have $\emptyset \not \in \Omega_\Lambda$, and since $\Lambda$ is uniformly discrete, every element of $\Omega_\Lambda$ is uniformly discrete and hence $G \not \in \Omega_\Lambda$. This shows that $\nu$ is non-trivial.

As explained in \cite[Sec. 5.1]{BH1} we have a well-defined continuous periodization map 
\[
 \mathcal P: C_c(G) \to C(\Omega_\Lambda), \quad \mathcal Pf(\Lambda') = \sum_{x \in \Lambda'} f(x).
\]
In particular we can define a Radon measure on $G$ by $\eta(f) := \nu(\mathcal Pf)$.
\begin{lemma} There exists $u \in L^1_{\rm loc}(G, m_G)$ such that $\eta = u \, m_G$.
\end{lemma}
\begin{proof} Since $[m_G]$ is the unique $G$-invariant measure class on $G$, it suffices to show that $\eta$ is $G$-quasi-invariant. This will follows from the fact that $\nu$ is $\mu$-stationary and hence $G$ quasi-invariant. To see this, let $K \subset G$ be a compact subset of positive measure with characteristic function $1_K$. There then exists a compact set $L \supset K$ and functions $f_n \in C_c(G)$ supported in $L$ such that $f_n \geq 1_K$ and $\liminf f_n = 1_K$. Then $\lim \eta(f_n) = \eta(K) > 0$, and hence for all sufficiently large $n$ we have 
\[
\nu(\mathcal P(f_n)) = \eta(f_n) >0
\]
 Since $\mathcal P$ is $G$-equivariant, $\nu$ is $G$-quasi-invariant and $\nu(\mathcal P(f_n))>0$ for all sufficiently large $n$, we have for all $g \in G$,
\[
\eta(gK) =\lim \eta(g.f_n) = \lim \nu(\mathcal P(g.f_n)) = \lim \nu(g. \mathcal P(f_n)) = \lim g_*\nu(\mathcal P(f_n)) > 0.
\]
Since $G$ is $\sigma$-compact, this proves that $\eta$ is $G$-quasi-invariant.
\end{proof}
It turns out that the density $u$ is $\mu$-stationary in the following sense; this statement includes in particular the fact that the convolution of $\mu$ with $u$ converges. 
\begin{lemma}\label{uStationary} For $m_G$-almost every $x \in G$ we have
\[
u(x)  \quad = \quad \int_G \rho(s)u(s^{-1}x) dm_G(s) \quad =\quad \int _G u(s^{-1}x) d\mu(s) \quad < \quad \infty.
\]
\end{lemma}
\begin{proof} Since $\nu$ is $\mu$-stationary and $\mathcal P$ is $G$-equivariant we have for every $f \in C_c(G)$.
\begin{eqnarray*}
\eta(f) &=& \nu(\mathcal Pf) \quad  = \quad \mu \ast \nu(\mathcal P f) \quad = \quad \int_G  \nu(g^{-1}.\mathcal Pf)\,d\mu(g)\\
&=&  \int_G  \nu(\mathcal P(g^{-1}.f))\,d\mu(g) \quad = \quad \int_G \eta(g^{-1}.f) d\mu(g).
\end{eqnarray*}
Since $\eta = u\, m_G$ we deduce from left-invariant of $m_G$ that
\[
m_G(f \cdot u) = \eta(f) = \int_G m_G((g^{-1}.f) \cdot u) d\mu(g) =  \int_G m_G(f \cdot (g.u)) d\mu(g).
\]
If $f \geq 0$, then we can apply Fubini to obtain
\begin{eqnarray*}
\int_G f(x) u(x)\, dm_G(x) &=& \int_G m_G(f \cdot (g.u)) d\mu(g)\\ &=& \int_G \left(\int_G f(x) u(g^{-1}x)dm_G(x)\right) d\mu(g)\\ &=& \int_G f(x) \left(\int_G u(g^{-1}x) d\mu(g)\right) dm_G(x).
\end{eqnarray*}
If $f \in C_c(G)$ is arbitrary, then we can write $f = f_+ - f_-$ with $f_+, f_- \geq 0$ and apply this formula to $f_+$ and $f_-$. The lemma follows.
\end{proof}
\begin{corollary} There exists a lower-semicontinuous positive function $v: G \to (0, \infty)$ such that $u(x) = v(x)$ for $m_G$-almost all $x \in G$. 
\end{corollary}
\begin{proof} Let $\rho_n$ be an increasing sequence in $C_c(G)$ with $\rho_n \nearrow \rho$. By monotone convergence and Lemma \ref{uStationary} we then have for almost all $x \in G$,
\[
u(x) = v(x) :=  \int_G \rho(s)u(s^{-1}x) dm_G(s) = {\rm sup}  \int_G \rho_n(s)u(s^{-1}x) dm_G(s).
\]
The integrals on the right hand side define continuous functions (as convolutions with $\rho_n \in C_c(G)$), hence $v$ is lower-semicontinuous as the supremum of continuous functions. We claim that the functon $v$ is strictly positive. Indeed, assume for contradiction that for some $x_0 \in G$ we would have
\[
v(x_0) = \int_G \rho(s)u(s^{-1}x_0) dm_G(s)  = 0
\]
Since the integrand is non-negative and $\rho>0$ this would imply that $u(x) = 0$ for $m_G$-almost every $x$, but then $\eta = u\, m_G = 0$, which is a contradiction.
\end{proof}
In view of the corollary we assume from now on that $u$ has been chosen to be positive and semicontinuous.
\begin{proof}[Proof of Theorem \ref{Unimodular}] We fix a compact set $K\subset G$ such that $G = \Lambda K = K\Lambda$. Since $u$ is lower semicontinuous and stricily positive we then have
\[
\delta := \inf_{k \in K} u(k) \delta(k) > 0.
\]
It follows from \cite[(5.2), p. 2957]{BH1} that there exists a finite set $F$ such that for all $\lambda \in \Lambda$ and $m_G$-almost all $g \in G$ we have
\[
u(g\lambda^{-1}) \Delta_G(\lambda^{-1}) \leq \sum_{c \in F} u(gc^{-1})\Delta_G(c^{-1}).
\]
Note that if $g$ satisfies this inequality and if we write $g = k\lambda$ with $k \in K$ and $\lambda \in \Lambda$, then
\[
\sum_{c \in F} u(gc^{-1})\Delta_G(c^{-1}) \geq u(g\lambda^{-1}) \Delta_G(\lambda^{-1}) = u(k)\Delta_G(k)\Delta_G(g^{-1}) \geq \delta \cdot \Delta_G(g)^{-1}.
\]
Since this holds for $m_G$-almost every $g\in G$, we deduce that for all $g\in G$ we have
\[
\int_G \left(\sum_{c \in F} u(s^{-1}gc^{-1})\Delta_G(c^{-1})\right) \rho^{\ast n}(s) dm_G(s)  \geq \delta \cdot \int_G  \Delta_G(s^{-1}g)^{-1}  \rho^{\ast n}(s) dm_G(s)
\]
By Lemma \ref{uStationary} the left-hand side equals
\[
\sum_{c \in F} \rho^{\ast n} \ast u(gc^{-1}) = \sum_{c \in C} u(gc^{-1}),
\]
whereas the right hand side equals to
\[
\delta \cdot \Delta_G(g)^{-1} \cdot \left(\int_G   \rho{^\ast n}(s) \delta(s) dm_G(s)\right) = \delta \cdot \Delta_G(g)^{-1} \cdot\left(\int_G \rho(s) \delta_G(s) \, dm_G(s)\right)^n,
\]
which diverges to $\infty$ by Property (iii) of Lemma \ref{LemmaUniNew}. We thus have established for every $g \in G$ that $\sum_{c \in C} u(gc^{-1}) = \infty$, contradicting the fact that $c$ is finite.
\end{proof}

\end{document}